\documentclass[draftcls, onecolumn, lettersize, 12pt, journal]{IEEEtran}
\usepackage[caption=false,font=normalsize,labelfont=sf,textfont=sf]{subfig}
\usepackage{textcomp}
\usepackage{stfloats}
\usepackage[T1]{fontenc}    
\usepackage{hyperref}       
\usepackage{url}            
\usepackage{booktabs}       
\usepackage{amsfonts}       
\usepackage{nicefrac}       
\usepackage{microtype}      
\usepackage{amsmath,amssymb}
\usepackage{graphicx}
\usepackage{textcomp}
\usepackage{xcolor}
\usepackage{amsthm,amssymb,bm, verbatim,dsfont,mathtools}
\usepackage[mathcal]{euscript}
\usepackage{color,graphicx}
\usepackage{etoolbox}
\usepackage{tikz}
\usepackage{xr,xspace}
\usepackage{soul}
\usepackage{accents}
\usepackage{comment}
\usepackage{cite}
\usepackage{csquotes}
\usepackage{algorithm}
\usepackage{algorithmic}
\usepackage{adjustbox}
\usepackage{multirow,array}

\newtheorem{theorem}{Theorem}

\newtheorem{claim}{Claim}
\newtheorem{remark}{Remark}

\newcommand{\identityf}[1]{\mathbf 1_{\{#1\}}}
\newlength{\dhatheight}

\newcommand{\mbf}[1]{\text{\boldmath{$#1$}}}

\begin{document}

\title{Robust Multi-Hypothesis Testing with Moment-Constrained Uncertainty Sets}

\author{Akshayaa Magesh, Zhongchang Sun, Venugopal V. Veeravalli, Shaofeng Zou
\thanks{A. Magesh and V.V. Veeravalli are with the University of Illinois at Urbana-Champaign, IL, 61820, USA. (email: amagesh2@illinois.edu, vvv@illinois.edu).
Z. Sun and S. Zou are with the University at Buffalo, the State University of New York, NY, 14260, USA. (email: {zhongcha@buffalo.edu}, {szou3@buffalo.edu}).}
\thanks{The work of Z. Sun and S. Zou was supported in part by the National Science Foundation
under Grants 2106560 and 2112693. The work of A. Magesh and V.V. Veeravalli was supported by the National Science Foundation under grant \#2106727.}
}



\maketitle

\begin{abstract}
The problem of robust multi-hypothesis testing in the Bayesian setting is studied in this paper. Under the $m \geq 2$ hypotheses, the data-generating distributions are assumed to belong to uncertainty sets constructed through some moment functions, i.e., the sets contain distributions whose moments are centered around empirical moments obtained from some training data sequences. The goal is to design a test that performs well under all distributions in the uncertainty sets, i.e., a test that minimizes the worst-case probability of error over the uncertainty sets. 
Insights on the need for optimization-based approaches to solve the robust testing problem with moment constrained uncertainty sets are provided. The optimal (robust) test based on the optimization approach is derived for the case where the observations belong to a finite-alphabet. When the size of the alphabet is infinite, the optimization problem is infinite-dimensional and intractable, and therefore a tractable finite-dimensional approximation is proposed, whose optimal value converges to the optimal value of the original problem as the size of the dimension {of the approximation} goes to infinity.
A robust test is constructed from the solution to the approximate problem, and guarantees on its worst-case error probability over the uncertainty sets are provided. Numerical results are provided to demonstrate the performance of the proposed robust test.
\end{abstract}


\section{Introduction}
\IEEEPARstart{H}{ypothesis} testing is a fundamental problem in statistical  decision-making, in which the goal is to decide between given hypotheses based on observed data. In the classical setting of binary hypothesis testing ($m = 2$), the two hypotheses are referred to as the null and alternative hypotheses. The likelihood ratio test between the two hypotheses is optimal {under various criteria (see, e.g., \cite{moulin2018statistical}).} The multi-hypothesis testing problem is an extension of binary hypothesis testing to $m \geq 2$ hypotheses. In this case, without loss of generality, the likelihood ratios of the $m$ hypotheses with respect to the first one is used to construct {optimal tests}. {Multi-hypothesis testing} problems have a wide range of applications in economics, communications, signal processing and life sciences. 

In general, the distributions under the hypotheses may be unknown, and may need to be estimated from historical data. However, deviations of the estimates from the true underlying distributions can result in significant performance degradation of the likelihood ratio test constructed using the estimated distributions. The \emph{robust} hypothesis testing framework was proposed in Huber's seminal work \cite{huber1965robust} to address this problem. In the robust setting, it is assumed that under each hypothesis, the distributions belong to certain uncertainty sets, and the goal is to build a detector that performs well under all distributions in the uncertainty sets. The uncertainty sets are generally constructed as collections of distributions that lie within some neighbourhood (for instance, with respect to some discrepancy measure) of certain \emph{nominal} distributions. 

The uncertainty sets in \cite{huber1965robust} were constructed as epsilon-contamination sets, i.e., each uncertainty set contains mixture distributions where the nominal distribution is corrupted by some unknown distribution. A censored likelihood ratio test was proposed and proved to be minmax optimal, {i.e., the test minimizes the worst-case error probability over the uncertainty sets.} The worst-case distributions, i.e., the distributions belonging to the uncertainty sets under which the error is maximized, were also characterized in \cite{huber1965robust} for epsilon-contamination sets. {Subsequently, minmax optimal tests through the characterization of the worst-case distributions were developed for uncertainty sets based on alternating capacities of order 2, Prokhorov neighbourhoods and bounded densities (see, e.g., \cite{rieder1977least,osterreicher1978on, bednarski1981on,hafner1982prokhorov, kassam1981robust,hafner1993construction,vastola1984ppoint})}. Moment-constrained uncertainty sets under the asymptotic Neyman-Pearson setting were studied in \cite{pandit2004asymptotic} for the case where the observations belong to a finite alphabet. 

In \cite{huber1965robust}, Huber additionally characterized a sufficient condition that needs to be satisfied by the worst-case distributions for general uncertainty classes in the binary hypothesis testing setting, which was later formalized as the Joint Stochastic Boundedness (JSB) {condition \cite{vvv1994minmax} (see also \cite{moulin2018statistical})}. This condition is quite powerful, as the optimal minmax test for a single sample or independent batch samples can be directly constructed from the worst-case distributions satisfying the JSB condition. {However, except for certain  structured uncertainty classes such as} epsilon-contamination classes, Total Variation (TV) constrained classes, classes with alternating capacities of order 2, and bounded density variations, it is quite difficult to verify the JSB condition. Alternate optimization based approaches to solve the robust hypothesis problem were {therefore} explored in \cite{levy2009robust, gul2017minmax, jin2020adversarial} for uncertainty classes constructed using f-divergences and adversarial perturbations.

Recent works have studied the problem of constructing uncertainty sets using a data-driven approach \cite{gao2018wasserstein, xie2021robust, wang2022data, sun2022kernel}, where the nominal distributions are the empirical distributions derived from sequences of training observations. Note that f-divergences are not useful in constructing uncertainty sets using a data-driven approach, since the resulting sets would contain only distributions supported on the empirical samples \cite{xie2021robust}. 
In \cite{gao2018wasserstein, xie2021robust}, the Wasserstein distance was used to construct the uncertainty sets {for a minmax robust Bayesian hypothesis testing problem}. The main result in \cite{xie2021robust} states that for the Wasserstein uncertainty sets, there exists a pair of worst-case distributions supported on the empirical samples. {Furthermore, a detection test using such a pair of worst-case distributions was proposed in \cite{xie2021robust}.} The work in \cite{wang2022data} studied the data-driven robust hypothesis testing problem with the uncertainty sets constructed using the Sinkhorn distance, where a test for minimizing the worst-case loss using an approximated smoothed error probability was proposed. In \cite{sun2022kernel}, the Maximum Mean Discrepancy (MMD) was used to construct the uncertainty sets. In the Bayesian setting, a tractable approximation to the minmax problem was proposed. Additionally, in the Neyman-Pearson setting, an asymptotically optimal test was proposed. The performances of the proposed tests in \cite{xie2021robust, wang2022data, sun2022kernel} were evaluated empirically, and theoretical guarantees on the worst-case error over the distributions in the uncertainty sets were absent. Note that all the above-mentioned works \cite{huber1965robust, rieder1977least,osterreicher1978on, bednarski1981on,hafner1982prokhorov, kassam1981robust,hafner1993construction,vastola1984ppoint, pandit2004asymptotic, levy2009robust, jin2020adversarial, gul2017minmax, gao2018wasserstein, xie2021robust, wang2022data, sun2022kernel} consider only the binary hypothesis testing problem. 

A related field of study with a rich literature is Distributionally Robust Optimization (DRO) \cite{rahimian2019distributionally}. The DRO problem has been studied with uncertainty sets constructed using f-divergences \cite{ben2013robust, hanasusanto2013robust,Zhang2024aaai}, Wasserstein distance \cite{mohajerin2018data, gao2018wasserstein, lee2018minmax}, Total Variation (TV) distance \cite{shapiro2017distributionally}, contamination classes \cite{duchi2019distributionally}, moment functions \cite{delage2010distributionally, xu2018distributionally} and kernel methods \cite{zhu2020kernel}. Methods developed for DRO problems are not directly applicable to robust hypothesis testing problems. Firstly, while there is only a single ambiguity set in DRO problems, there are $m\geq2$ such sets in robust {hypothesis testing} problems. This causes most algorithms that use first-order methods to solve DRO problems to be inapplicable in the robust {hypothesis} testing. Secondly, the decision variable is assumed to be a finite-dimensional vector in DRO problems, while the decision function or the detection test is generally infinite dimensional in the robust hypothesis testing, which makes {the latter problem intractable in general}. Thus,  it is of interest to develop optimization based approaches to solve the robust hypothesis testing problem.

\subsection{Our Contributions}
In this paper, we study the problem of robust multi-hypothesis testing with general moment-constrained uncertainty sets, i.e., the sets contain distributions whose moments are centered around empirical moments. 
Mean constrained and variance constrained uncertainty sets can be viewed as special cases of this setting. Moment constrained sets are practical since it is generally computationally easier to calculate empirical moments, and the convergence rates  of empirical moments to the true moments are faster than that of the Wasserstein or MMD distances between the empirical distributions and the underlying true distributions. Moment functions have been extensively used in statistics to model distributions, and constructing uncertainty sets using multiple moment functions provides higher flexibility in modelling the robust problem. Our focus is on the minmax robust formulation in the Bayesian setting. {To the best of our knowledge, ours is the first work to address the robust multi-hypothesis testing setting for more than two hypotheses.}

\begin{itemize}
    \item In the basic case of robust binary hypothesis testing ($m = 2$), we provide insights on the need for optimization based approaches in solving the moment constrained robust hypothesis testing problem by showing that the moment-constrained uncertainty sets do not satisfy the JSB condition. This provides a justification for studying optimization based approaches to solve the problem. 
    \item We characterize the minmax optimal robust test for the case when the distributions under the $m$ hypotheses are supported on a finite alphabet set $\mathcal{X}$. 
    \item We then extend the study to the case when the size of $\mathcal{X}$ is infinite\footnote{The infinite-alphabet case includes both the continuous-alphabet and discrete infinite-alphabet cases.}. In this case, we provide a tractable approximation of the worst-case error that converges to the optimal minmax Bayes error, and propose a robust test that generalizes to the entire alphabet.
    \item We provide robust guarantees on the worst-case error of the proposed test over all distributions in the uncertainty sets.  To the best of our knowledge, ours is the first work to provide such guarantees. 
    \item  We provide numerical results to demonstrate the performance of our proposed algorithms. 
\end{itemize}

Some preliminary results that considered only the binary hypothesis testing setting were presented at \cite{magesh2023robust}.

\section{Problem Setup}

Let the random variable $X$ denote a single observation, and let $\mathcal{X} \subset \mathbb{R}^d$ be a compact set denoting the corresponding sample space, where $d$ is the dimension of the data. 
Let $\mathcal{P}$ denote the set of all Borel probability measures on $\mathcal{X}$. Let $\mathcal{P}_1, \ldots, \mathcal{P}_m \subset \mathcal{P}$ denote the uncertainty sets under the $m$ hypotheses, respectively. We construct the uncertainty sets using general moment constraints derived  from observations from the hypotheses. In this paper, we use the notation $[n]$ to denote the set $\{1, \ldots, n\}$ for any $n \in \mathcal{N}$. Let $\mbf{\hat{x}_i} = (\hat{x}_{i,1}, \ldots, \hat{x}_{i,n_i})$ for $i \in [m]$
denote the realizations of training sequences under the $m$ hypotheses. Let $\hat{Q}_1, \ldots, \hat{Q}_m$ denote the empirical distributions corresponding to the training observations, i.e., 
\[
\hat{Q}_i = \frac{1}{n_i}\sum_{j=1}^{n_i}\delta(\hat{x}_{i,j}),
\]
where $\delta(x)$ is the Dirac measure on $x$.
We use the empirical distributions as the nominal distributions in the construction of the uncertainty sets. Let $$\psi_k : \mathcal{X} \to \mathbb{R};\; k \in [K]$$ denote $K$ real-valued, continuous functions defined on the sample space, where $[K]=\{1,...,K\}$. The uncertainty sets for $i \in [m]$ are defined as follows:
\begin{equation}\label{eq:uncertainty_sets_m-ary}
    \mathcal{P}_i^\eta = \left\{P \in \mathcal{P} : \left| E_{P}[\psi_k(X)] - E_{\hat{Q}_i}[{\psi_k(X)}] \right| \leq \eta, \;\;  k \in [K] \right\}, 
\end{equation}
where $\eta$ is a pre-specified radius of the uncertainty sets, and $E_{P}[\cdot]$ denotes the expectation under the distribution $P$. Note that we can choose the same $\eta$ for all the constraints, since otherwise the moment functions $\psi_k$'s can be scaled appropriately. 
Let $\eta_\mathrm{max}$ denote the maximum radius so that the uncertainty sets do not overlap. Thus, it is assumed that $\eta \in (0, \eta_\mathrm{max})$, otherwise, the problem becomes ill-defined.
For instance, in the case where $m = 2$, 
\[\eta < \mathop{\max}\limits_{k \in [K]} \frac{\left|E_{\hat{Q}_2}[\psi_k(X)] - E_{\hat{Q}_1}[\psi_k(X)] \right|}{2}.
\]
The $m$ hypotheses with uncertainty sets $\mathcal{P}_1^\eta, \ldots, \mathcal{P}_m^\eta$
are defined as follows:
\begin{align} \label{eq:def_robust_test}
    &H_1 : X \sim P_1, \;\;\; P_1 \in \mathcal{P}^\eta_1, \nonumber\\
    &\hspace{1cm}\vdots \hspace{2.5cm}\vdots \nonumber\\
    &H_m : X \sim P_m, \;\;\; P_m \in \mathcal{P}^\eta_m.
\end{align}

A (randomized) decision rule is defined using $\phi(x) = (\phi_1(x), \ldots,  \phi_m(x))$, a measurable function from $\mathcal{X}$ to the $m$-dimensional probability simplex $\Delta_m$, 
with the interpretation that the test accepts hypothesis $i$ with probability $\phi_i(x)$ (see Chapter 1 of \cite{berger2013statistical}). 
Note that $\phi_i(x) \geq 0$ and $\sum_{i=1}^m \phi_i(x) = 1$ for all $x \in \mathcal{X}$. We assume uniform costs, i.e., the cost of the test declaring $H_j$ when $H_i$ is true is $\identityf{j \neq i}$. The conditional risk $R_i(\phi)$ associated with hypothesis $i$, i.e., the probability of the test rejecting $H_i$ when $H_i$ is true, is given by 
\begin{equation}
    R_i(\phi) = \int_\mathcal{X} \left(1 - \phi_i(x)\right) dP_i(x).
\end{equation}
In the Bayesian setting with equal priors, the probability of error of the test  is given by:
\begin{flalign}\label{eq:Bayes_error}
P_E(\phi; P_1, \ldots, P_m) &\triangleq \sum_{i=1}^m \frac{1}{m} R_i(\delta) \nonumber\\
&= 1 - \frac{1}{m} \sum_{i=1}^m \int_\mathcal{X} \phi_i(x) dP_i(x).
\end{flalign}

The goal is to solve
\begin{flalign}\label{eq:minmax_test_m_ary}
\inf_{\phi}\sup_{P_i\in\mathcal{P}^\eta_i; i \in [m]} P_E(\phi; P_1, \ldots, P_m).
\end{flalign}
Note that in order to minimize the Bayes probability of error in \eqref{eq:Bayes_error} for fixed $P_1, \ldots, P_m$, it is not necessary to consider randomized decision rules. This can be seen from the fact that $P_E(\phi; P_1, \ldots, P_m) \geq 1 - \frac{1}{m} \int_\mathcal{X} \max\limits_{i \in [m]} \left\{ p_i(x) \right\} d\mu(x)$, where $p_i$ are densities of distributions $P_i$ with respect to some reference measure $\mu$, and this lower bound is achieved by the deterministic test that chooses the hypothesis with maximum likelihood ratio $\frac{p_i(x)}{p_1(x)}$ and breaks ties arbitrarily. 

\begin{remark}
The methods proposed in this paper and their analyses can be easily extended to the setting with unequal priors.
\end{remark}

\section{Insights from the Binary Case}

Consider the robust binary hypothesis setting, for which $m = 2$. The minmax robust hypothesis testing problem can equivalently be written as \footnote{Only for the binary case, we index the hypotheses by $i = 0,1$ rather than $i = 1,2$ to be consistent with literature on binary hypothesis testing.}:
\begin{flalign}\label{eq:min_max_err_binary}
\inf_{\phi}\sup_{P_0\in\mathcal{P}_0, P_1\in\mathcal{P}_1} P_E(\phi; P_0, P_1),
\end{flalign}
where 
\begin{flalign}\label{eq:min_max_err}
P_E(\phi; P_0, P_1) \triangleq \frac{1}{2}E_{P_0} \big[\phi(x)\big] + \frac{1}{2}E_{P_1} \big[1-\phi(x)\big],
\end{flalign}
and a test $\phi:\mathcal{X}\rightarrow [0,1]$ accepts $H_0$ with probability $1- \phi(x)$, and accepts $H_1$ with probability $\phi(x)$.
Much of the prior work in robust hypothesis testing has been focused on the binary setting. An important result in this setting from Huber's seminal work in \cite{huber1965robust} is the characterization of sufficient conditions for a pair of least favorable distributions for the construction of a saddle point, later formalized in \cite{vvv1994minmax} as the \emph{Joint Stochastic Boundedness (JSB)} condition. This condition states that, a pair of uncertainty classes $\mathcal{P}_0$ and $\mathcal{P}_1$ is jointly stochastically bounded by $(Q_0, Q_1)$, if there exist distributions $Q_0 \in \mathcal{P}_0$ and $Q_1 \in \mathcal{P}_1$, such that for any distributions $P_0 \in \mathcal{P}_0$ and $P_1 \in \mathcal{P}_1$, and all $t \in \mathbb{R}$, 
\begin{equation}\label{eq:JSB_0}
    P_0\{\ln L_q(x) > t\} \leq  Q_0\{\ln L_q(x) > t\}, 
\end{equation}
and
\begin{equation}\label{eq:JSB_1}
    P_1\{\ln L_q(x) > t\} \geq  Q_1\{\ln L_q(x) > t\}, 
\end{equation}
where $L_q(x) = \frac{q_1(x)}{q_0(x)}$ is the likelihood ratio between $Q_1$ and $Q_0$. Intuitively, this condition leads to a pair of distributions from the uncertainty sets that are closest to each other {for the purpose of hypothesis testing.}

The advantage of {having} a pair of worst-case distributions that satisfies the JSB condition is that the optimal solution for the minmax Bayes problem can be constructed directly as the likelihood ratio test between this pair of distributions. Additionally, for testing a batch of independent {observations}, the sum of the log-likelihood ratios can be used to construct the optimal test. However, except for a few special cases of the uncertainty classes such as the epsilon-contamination classes and total variation constrained classes (see \cite{huber1965robust}), it is difficult to ascertain the existence of distributions satisfying the JSB condition. Most {recent} works in the robust hypothesis testing literature employ optimization-based approaches, and {do not even} discuss the JSB property \cite{gul2017minmax, xie2021robust, wang2022data, sun2022kernel}. 

For moment-constrained uncertainty classes, we provide some insights on the existence of distributions satisfying the JSB property. In particular, we provide a counterexample to show that the JSB property may not be satisfied in a simple finite alphabet setting, where $\mathcal{P}_0$ contains a single distribution, and $\mathcal{P}_1$ contains distributions satisfying a single first-order moment constraint.

\begin{claim}\label{claim:jsb}
    Consider the alphabet $\mathcal{X} =   \{0,1,2\}$. Let $\mathcal{P}_0$ contain a single distribution $P_0 = [0.45, 0.25, 0.3]$. Let the uncertainty set $\mathcal{P}_1$ be defined as:
\begin{equation}
    \mathcal{P}_1 = \{P \in \mathcal{P} : E_P[X] \geq 0.9\}.
\end{equation}
    This pair of uncertainty sets defined on $\mathcal{X}$ does not satisfy the JSB condition in \eqref{eq:JSB_0} and \eqref{eq:JSB_1}.
\end{claim}

\begin{proof}
    In order to prove that this pair of uncertainty sets does not satisfy the JSB condition, we need to show that there does not exist any $Q_0 \in \mathcal{P}_0$ and $Q_1 \in \mathcal{P}_1$ such that the conditions in \eqref{eq:JSB_0} and \eqref{eq:JSB_1} are satisfied. It follows from the structure of $\mathcal{P}_0$ that the only possibility for $Q_0$ is $P_0$. Among the distributions in $\mathcal{P}_1$, it suffices to check for distributions $Q_1 \in \mathcal{P}_1$ that satisfy the moment constraint with equality, since any $(Q_0, Q_1)$ that satisfies the JSB condition is also part of a saddle point solution for the minmax problem 
    \cite{moulin2018statistical}
    , i.e., $Q_1$ needs to satisfy
\begin{equation}
    0*Q_1(0) + 1*Q_1(1) + 2*Q_1(2) = 0.9. 
\end{equation}
Since $Q_1$ also needs to be a valid distribution, $Q_1(0) = 1 - Q_1(1) - Q_1(2)$. Using this with the above condition, we have that 
\begin{equation*}
    \begin{split}
         Q_1(2) &= q \\
        Q_1(0) &= 0.1 + q \\
        Q_1(1) &= 0.9 - 2q 
    \end{split}
\end{equation*}
and 
\begin{equation*}
    \begin{split}
        L_q(0) &= \frac{0.1 + q}{0.45} \\
        L_q(1) &= \frac{0.9 - 2q}{0.25} \\
        L_q(2) &= \frac{q}{0.3}
    \end{split}
\end{equation*}
{for $q \in [0,1]$.
Thus, we have reduced the problem to showing that the conditions in \eqref{eq:JSB_0} and \eqref{eq:JSB_1} are not satisfied for all $q \in [0,1]$.} We consider the following cases:
\begin{itemize}
    \item \textbf{Case 1 -} $q \in [0,0.2]$:
    In this case, $L_q(1) \geq L_q(0) \geq L_q(2)$. Thus, the possibilities for the event $\{\ln{L_q(x)} > t\}$ 
    are $\{2,0,1\} , \{1,0\}, \{1\}, \emptyset $ 
    for $t \in \mathbb{R}$. Note that 
    \begin{equation*}
    \begin{split}
        &Q_1(\{2,0,1\}) = 1 \\
        &Q_1(\{1,0\}) \in [0.8,1] \\
        &Q_1(\{1\}) \in [0.5,0.9] \\
        &Q_1(\emptyset) = 0.
    \end{split}
    \end{equation*}
    Consider the distribution $P_1 = [0.05, 0.65, 0.3] \in \mathcal{P}_1$. There exists no possible $Q_1$ such that \eqref{eq:JSB_0} and \eqref{eq:JSB_1} are satisfied. Thus, $q \in [0,0.2]$ does not contain a possible solution. 

    \item \textbf{Case 2 -} $q \in [0.2,0.3176]$: In this case, $L_q(1) \geq L_q(2) \geq L_q(0)$,  and the possibilities for the event for $\{L_q(x) > t\}$ are $\{1,2,0\} , \{1,2\}, \{1\}, \emptyset $ for $t \in \mathbb{R}$. Note that 
    \begin{equation*}
    \begin{split}
        &Q_1(\{1,2,0\}) = 1 \\
        &Q_1(\{1,2\}) \in [0.5824,0.7] \\
        &Q_1(\{1\}) \in [0.2647,0.5] \\
        &Q_1(\emptyset) = 0.
    \end{split}
    \end{equation*}
    Consider the distribution $P_1 = [0.3, 0.2, 0.5] \in \mathcal{P}_1$. There exists no possible $Q_1$ such that \eqref{eq:JSB_0} and \eqref{eq:JSB_1} are satisfied. Thus, $q \in [0.2,0.3176]$ does not contain a possible solution. 

    \item \textbf{Case 3 -} $q \in [0.3176, 0.3304]$: In this case, $L_q(2) \geq L_q(1) \geq L_q(0)$,  and the possibilities for the event for $\{L_q(x) > t\}$ are $\{0,1,2\} , \{1,2\}, \{2\}, \emptyset$ for $t \in \mathbb{R}$. Note that 
    \begin{equation*}
    \begin{split}
        &Q_1(\{0,1,2\}) = 1 \\
        &Q_1(\{1,2\}) \in [0.5696, 0.5824] \\
        &Q_1(\{2\}) \in [0.3176,0.3304] \\
        &Q_1(\emptyset) = 0.
    \end{split}
    \end{equation*}
    Consider the distribution $P_1 = [0.1, 0.6, 0.3] \in \mathcal{P}_1$. There exists no possible $Q_1$ such that \eqref{eq:JSB_0} and \eqref{eq:JSB_1} are satisfied. Thus, $q \in [0.3176, 0.3304]$ does not contain a possible solution. 

    \item \textbf{Case 4 -} $q \in [0.3304, 0.45]$: In this case, $L_q(2) \geq L_q(0) \geq L_q(1)$,  and the possibilities for the event for $\{L_q(x) > t\}$ are $\{0,1,2\} , \{0,2\}, \{2\}, \emptyset $ for $t \in \mathbb{R}$. Note that 
    \begin{equation*}
    \begin{split}
        &Q_1(\{0,1,2\}) = 1 \\
        &Q_1(\{0,2\}) \in [0.7608,1] \\
        &Q_1(\{2\}) \in [0.3304, 0.45] \\
        &Q_1(\emptyset) = 0.
    \end{split}
    \end{equation*}
    Consider the distribution $P_1 = [0.1, 0.6, 0.3] \in \mathcal{P}_1$. There exists no possible $Q_1$ such that \eqref{eq:JSB_0} and \eqref{eq:JSB_1} are satisfied. Thus, $q \in [0.3304, 0.45]$ does not contain a possible solution. 

    \item \textbf{Case 5 -} $q \in [0.45, 1]$: This case is not possible, since for these values of $q$, $Q_1(1)$ does not take valid probability values.
    \end{itemize}

    Thus, we have that there exist no possible $Q_0 \in \mathcal{P}_0$ and $Q_1 \in \mathcal{P}_1$ that satisfy the JSB condition. 
\end{proof}

Through the construction of a counterexample, we have seen that the JSB {condition may} not hold for moment-constrained uncertainty classes in general. This motivates using an optimization based approach in this paper. 

\section{Robust Multi-Hypothesis Testing}

First, we present a result to establish the existence of a saddle point for our minmax problem. Then, we characterize the minmax optimal test in the case where the alphabet is finite. In the infinite-alphabet case, {for which} the optimization problem is infinite dimensional, we present a tractable finite-dimensional approximation, and bound 
the error due to the approximation. Henceforth, we denote $E[\psi_k(X)]$ by $E[\psi_k]$ for ease of notation.

\begin{theorem}\label{theorem:sion}
The minmax problem in \eqref{eq:minmax_test_m_ary} has a saddle-point solution $(\phi^*; P_1^*, \ldots, P_m^*)$, i.e., 
\begin{equation}\label{eq:minmax_maxmin}
    \inf_{\phi}\sup_{P_i\in\mathcal{P}^\eta_i; i \in [m]} P_E(\phi; P_1, \ldots, P_m) = \sup_{P_i\in\mathcal{P}^\eta_i; i \in [m]} \inf_{\phi} P_E(\phi; P_1, \ldots, P_m).
\end{equation}
\end{theorem}

\begin{proof}
The uncertainty sets $\mathcal{P}^\eta_i$ are convex. Indeed, for some $\lambda \in (0,1)$, $P, Q \in \mathcal{P}^\eta_i$ for $i \in [m]$, and $k \in [K]$, we have that 
\begin{align}
    &\left| E_{\lambda P + (1- \lambda) Q} \left[ \psi_k \right] - E_{\hat{Q}_i}\left[ \psi_k \right] \right| \nonumber\\ 
    &=  \left| \lambda E_{ P} \left[ \psi_k \right]  + (1- \lambda) E_{ Q} \left[ \psi_k \right] - E_{\hat{Q}_i}\left[ \psi_k \right] \right| \nonumber\\
    &\leq \lambda \left| E_{P} \left[ \psi_k \right] - E_{\hat{Q}_i}\left[ \psi_k \right] \right| + (1-\lambda) \left| E_{Q} \left[ \psi_k \right] - E_{\hat{Q}_i}\left[ \psi_k \right] \right| \nonumber\\
    &\leq \eta.
\end{align}
Thus, $\lambda P + (1- \lambda) Q \in \mathcal{P}^\eta_i$, and we have that  $\mathcal{P}^\eta_i$ are convex for $i \in [m]$.
Note that it follows that $\mathcal{P}^\eta_1 \times \ldots \times \mathcal{P}^\eta_m$ is also convex. Since $\mathcal{X}$ is compact, we have that the set of all decision rules is compact and convex. Finally, the error probability $P_E(\phi; P_1, \ldots, P_m)$ is continuous, real-valued and linear in $\phi, P_1, \ldots, P_m$. Thus, applying Sion's minmax theorem \cite{sion1958minmax} concludes the proof.
\end{proof}


\subsection{Finite Alphabet: Optimal Test}

Consider a finite alphabet $\mathcal{X} = \{ z_1, \ldots, z_N\}$. Let $P_{i,N}$ denote a probability mass function on $\mathcal{X}$. The randomized decision rule for a finite alphabet $\mathcal{X}$ can be written as $\phi_{N} = [\phi_{1,N}, \ldots, \phi_{m,N}]^T$, where $\phi_{i,N} = [\phi_{i,N}(z_1), \ldots, \phi_{i,N}(z_N)]^T$ and $\sum_{i=1}^m \phi_{i,N}(z_j) = 1$ for $j \in [N]$. Then the minmax problem in \eqref{eq:minmax_test_m_ary} can be written as:
\begin{equation}\label{eq:minmax_finite_m_ary}
\begin{split}
         \min_{\phi_{i,N} \in [0,1]^{N}, i \in [m]} &\max_{P_{i,N} \in [0,1]^N, i \in [m]}  1 - \frac{1}{m} \sum_{i=1}^m \phi_{i,N}^T P_{i,N}  \\ 
        \text{s.t. }\;\; & \left| \sum_{j=1}^N p_{i,N}(z_j) \psi_k (z_j) - E_{\hat{Q}_i}[{\psi_k}] \right| \leq \eta,\;\; k \in [K]\\
        &  P_{i,N}^T \mathbf{1}= 1, \;\; i \in [m] \\
        & \sum_{i=1}^m \phi_{i,N}(z_j) = 1, \;\; j \in [N], 
        \end{split}
\end{equation}
where $P_{i,N} = [p_{i,N}(z_1), \ldots, p_{i,N}(z_N)]^T$ for $i \in [m]$ and $\mathbf{1}$ is the vector of all ones.     
Using the dual formulation for the inner maximization, we can reformulate the minmax problem into a tractable minimization problem. 

\begin{theorem}\label{thm:m_ary_finite_dual}
    The minmax optimization problem in \eqref{eq:minmax_finite_m_ary} can be reformulated as:
    \begin{equation}\label{eq:minmax_finite_dual_m_ary}
    \begin{split}
            &\min_{\substack{\phi_{i,N} \in [0,1]^{N}, i \in [m], \\ \lambda^\ell_{i,k},\lambda^u_{i,k} \geq 0, k \in [K], i\in [m] \\ \mu_1, \ldots, \mu_m \in \mathbb{R}}} 1 - \sum_{i=1}^m\sum_{k=1}^K \lambda_{i,k}^\ell a_{0,k}  + \sum_{i=1}^m\sum_{k=1}^K \lambda_{i,k}^u b_{0,k}  - \sum_{i=1}^m \mu_i, \\
            \text{s.t. }\;\; & \frac{1}{m}\phi_{i,N}(z_j) - \sum_{k=1}^K \lambda^\ell_{i,k} \psi_{k}(z_j) + \sum_{k=1}^K \lambda^u_{i,k}\psi_{k}(z_j) - \mu_i \geq 0,  \;\;\; j \in [N], 
        \;\;\; i \in [m] \\
            & \sum_{i=1}^m \phi_{i,N}(z_j) = 1, 
             \;\;\; j \in [N],
             \end{split}
    \end{equation}
    where $a_{i,k} = E_{\hat{Q}_i}[{\psi_k}] - \eta$, $b_{i,k} = E_{\hat{Q}_i}[{\psi_k}] + \eta$.
\end{theorem}

\begin{proof}
    For some fixed $\phi_N$, consider the inner maximization problem in \eqref{eq:minmax_finite_m_ary}, which can be equivalently written as
    \begin{equation}\label{eq:max_finite_m_ary}
    \begin{split}
         &\min_{P_{i,N} \in [0,\infty)^N, i \in [m]}  - 1 + \frac{1}{m} \sum_{i=1}^m \phi_{i,N}^T P_{i,N}  \\ 
        \text{s.t. }\;\; & \left| \sum_{j=1}^N p_{i,N}(z_j) \psi_k (z_j) - E_{\hat{Q}_i}[{\psi_k}] \right| \leq \eta,\;\; k \in [K] \\
        & P_{i,N}^T \mathbf{1} = 1, \;\; i \in [m].  
        \end{split}
\end{equation}
    We can write the Lagrangian of the above problem as
    \begin{align}\label{eq:finite_lagrangian}
      L(P_{1,N}, \ldots, P_{m,N}; \mathbf{\lambda}_1, \ldots, \mathbf{\lambda}_m, \mathbf{\mu}) &= -1 +  \frac{1}{m} \sum_{i=1}^m \phi_{i,N}^T P_{i,N}
      + \sum_{i=1}^m \sum_{k=1}^K \lambda_{i,k}^\ell (a_{i,k} - \mathbf{\psi_k}^T P_{i,N}) \nonumber\\
      &+ \sum_{i=1}^m \sum_{k=1}^K \lambda_{i,k}^u (\mathbf{\psi_k}^T P_{i,N} - b_{i,k}) 
      + \sum_{i=1}^m \mu_i(1 - \mathbf{1}^T P_{i,N}),
\end{align}
or equivalently, 
\begin{align}\label{eq:finite_lagrangian_2}
    L(P_{1,N}, \ldots, P_{m,N}; \mathbf{\lambda}_1, \ldots, \mathbf{\lambda}_m, \mathbf{\mu})
     &=  \sum_{i=1}^m P_{i,N}^T \left( \frac{1}{m} \phi_{i,N} - \sum_{k=1}^K \lambda_{i,k}^\ell \mathbf{\psi_k} + \sum_{k=1}^K \lambda_{i,k}^u \mathbf{\psi_k} -\mu_i \mathbf{1} \right) \nonumber\\
     & -1 + \sum_{i=1}^m \sum_{k=1}^K \lambda_{i,k}^\ell a_{i,k} - \sum_{i=1}^m  \sum_{k=1}^K \lambda_{i,k}^\ell b_{i,k} 
     +  \sum_{i=1}^m \mu_i,
\end{align}
where $\mathbf{\lambda}_i = (\lambda^\ell_{i,1}, \ldots, \lambda^\ell_{i,K}, \lambda^u_{i,1}, \ldots, \lambda^u_{i,K}) $ for $i \in [m]$ and $\mathbf{\mu} = (\mu_1, \ldots, \mu_m)$ are the dual variables. 

It is clear to see that the dual function is 
\begin{align}\label{eq:dual_function_finite}
        &\min_{P_{i,N} \in [0,\infty)^N, i \in [m]} L(P_{1,N}, \ldots, P_{m,N}; \mathbf{\lambda}_1, \ldots, \mathbf{\lambda}_m, \mathbf{\mu}) \nonumber\\
       & = - 1 + \sum_{i=1}^m\sum_{k=1}^K \lambda_{i,k}^\ell a_{0,k}  - \sum_{i=1}^m\sum_{k=1}^K \lambda_{i,k}^u b_{0,k}  + \sum_{i=1}^m \mu_i,
\end{align}
with 
\begin{align}
        &\frac{1}{m}\phi_{i,N}(z_j) - \sum_{k=1}^K \lambda^\ell_{i,k} \psi_{k}(z_j) + \sum_{k=1}^K \lambda^u_{i,k}\psi_{k}(z_j) - \mu_i \geq 0, \;\;\; j \in [N], 
        \;\;\; i \in [m], \\
            & \sum_{i=1}^m \phi_{i,N}(z_j) = 1, \;\;\; j \in [N].    
\end{align}
Since the objective function of the inner maximization problem is linear in $P_{1,N}, \ldots, P_{m,N}$, and the feasible set is convex, strong duality holds by Prop 6.2.1 in \cite{bertsekas2016nonlinear}. Thus, using the dual function for the inner maximization problem as in \eqref{eq:dual_function_finite}, we can write the dual problem of the original minmax problem as in \eqref{eq:minmax_finite_dual_m_ary}. 
\end{proof}

By solving the optimization problem in \eqref{eq:minmax_finite_dual_m_ary}, we get the optimal minmax decision rule $\phi_N^*$ for all points in the finite alphabet $\mathcal{X}$. Let $\gamma_N$ denote the optimal value of \eqref{eq:minmax_finite_dual_m_ary}. Note that for any test $\phi_N$, the saddle point solution $(\phi_N^*, P_{1,N}^*, \ldots,  P_{m,N}^*)$ satisfies
\begin{equation}
    P_E(\phi_N^*, P_{1,N}^*, \ldots,  P_{m,N}^*) \leq P_E(\phi_N, P_{1,N}^*, \ldots,  P_{m,N}^*), 
\end{equation}
i.e., 
\begin{equation}\label{eq:finite_saddle_property}
    1 - \frac{1}{m} \sum_{i=1}^m \sum_{j=1}^N {\phi_{i,N}^*}(z_j) p^*_{i,N}(z_j) \leq 
    1 - \frac{1}{m} \sum_{i=1}^m \sum_{j=1}^N \phi_{i,N}(z_j) p^*_{i,N}(z_j).
\end{equation}
The test $\phi_N^*$ has to be in the form of a likelihood ratio test constructed using $P_{1,N}^*, \ldots,  P_{m,N}^*$ in case of no tie breaks. Indeed, for any $z_j$ with $\mathop{\arg\max}\limits_{i 
\in [m]} \frac{p_{i,N}^*(z_j)}{p_{1,N}^*(z_j)} = 
\{i^*\}$, if the optimal test is such that $\phi_{i^*,N}^*(z_j) \neq 1$, it is clear to see from \eqref{eq:finite_saddle_property} that another test with $\phi_{i^*,N}(z_j) = 1$ will have a lower probability of error, resulting in a contradiction. 

We can obtain the corresponding saddle point worst-case distributions $(P_{1,N}^*, \ldots,  P_{m,N}^*)$ by substituting $\phi_N^*$ in the original minmax objective \eqref{eq:minmax_finite_m_ary} with an additional constraint to satisfy the saddle point property $\inf_{\phi_N} P_E(\phi_N, P_{1,N}^*, \ldots,  P_{m,N}^*) = \gamma_N$, and solving the corresponding maximization problem:
\begin{equation}
    \begin{split}
     &\max_{P_{i,N} \in [0,1]^N, i \in [m]}  1 - \frac{1}{m} \sum_{i=1}^m {\phi^*_{i,N}}^T P_{i,N}  \\ 
        \text{s.t. }\;\; & \left| \sum_{j=1}^N p_{i,N}(z_j) \psi_k (z_j) - E_{\hat{Q}_i}[{\psi_k}] \right| \leq \eta,\;\; k \in [K]\\
        & P_{i,N}^T \mathbf{1} = 1, \;\; i \in [m] \\
        &  1 - \frac{1}{m} \sum_{j=1}^N  \max_{i \in [m]} \{p_{i,N}(z_j)\} = \gamma_N.  
    \end{split}
\end{equation}

\subsection{Infinite Alphabet}

In this section, we consider the case when $\mathcal{X}$ is infinite.  In this case, the minmax formulation in \eqref{eq:minmax_test_m_ary} is in general an infinite-dimensional optimization problem, and closed form solutions are difficult to derive.
Indeed, the minmax problem is
\begin{equation}\label{eq:minmax_infinite_m_ary}
    \begin{split}
         \inf_{\phi} &\sup_{P_i \in \mathcal{M}, i \in [m]}  1 - \frac{1}{m} \sum_{i=1}^m \int_\mathcal{X} \phi_i(x) dP_i(x)  \\ 
        \text{s.t. }\;\; & \left| \int_\mathcal{X}  \psi_k (x) dP_i(x) - E_{\hat{Q}_i}[{\psi_k}] \right| \leq \eta,\;\; k \in [K], \;\; i \in [m] \\
        & \int_\mathcal{X} dP_i(x) = 1, \;\; i \in [m].  
    \end{split}
\end{equation}
where $\mathcal{M}$ is the set of all positive signed measures. We can characterize the dual for this problem with a form similar to the optimization problem in Theorem \ref{thm:m_ary_finite_dual}.

\begin{theorem}\label{thm:m_ary_infinite_dual}
    The strong dual for the minmax robust problem in the infinite alphabet case is 
    \begin{equation}\label{eq:minmax_infinite_dual_m_ary}
        \begin{split}
            &\inf_{\substack{\phi, \\ \lambda^\ell_{i,k},\lambda^u_{i,k} \geq 0, k \in [K], i\in [m] \\ \mu_1, \ldots, \mu_m \in \mathbb{R}}} 1 - \sum_{i=1}^m\sum_{k=1}^K \lambda_{i,k}^\ell a_{0,k} + \sum_{i=1}^m\sum_{k=1}^K \lambda_{i,k}^u b_{0,k}  - \sum_{i=1}^m \mu_i, \\
            \text{s.t. }\;\; & \frac{1}{m}\phi_i(x) - \sum_{k=1}^K \lambda^\ell_{i,k} \psi_{k}(x) + \sum_{k=1}^K \lambda^u_{i,k}\psi_{k}(x) - \mu_i \geq 0, \;\; \forall x \in \mathcal{X}, \;\; i \in [m],   \\
            & \sum_{i=1}^m \phi_i(x) = 1,  \;\; \forall x \in \mathcal{X},
        \end{split}
    \end{equation}
    where $a_{i,k} = E_{\hat{Q}_i}[{\psi_k}] - \eta$, $b_{i,k} = E_{\hat{Q}_i}[{\psi_k}] + \eta$.
\end{theorem}

\begin{proof}
     Similar to Theorem \ref{thm:m_ary_finite_dual}, we construct the Lagrangian for the inner maximization (or equivalently minimization of the negative objective) in \eqref{eq:minmax_infinite_m_ary} as follows:
     \begin{align}\label{eq:infinite_lagrangian}
      &L(P_1, \ldots, P_m; \mathbf{\lambda}_1, \ldots, \mathbf{\lambda}_m, \mathbf{\mu}) \nonumber\\
      &= -1 +  \frac{1}{m} \sum_{i=1}^m \int_\mathcal{X} \phi_i(x) dP_i(x)
      + \sum_{i=1}^m \sum_{k=1}^K \lambda_{i,k}^\ell \left( a_{i,k} - \int_\mathcal{X} {\psi_k}(x) dP_i(x) \right)\nonumber\\
     &\quad
     + \sum_{i=1}^m \sum_{k=1}^K \lambda_{i,k}^u \left( \int_\mathcal{X} {\psi_k(x)}^T dP_i(x) - b_{i,k} \right)  + \sum_{i=1}^m \mu_i \left( 1 - \int_\mathcal{X} dP_i(x) \right),
\end{align}
or equivalently, 
\begin{align}\label{eq:infinite_lagrangian_2}
    L(P_1, \ldots, P_m; \mathbf{\lambda}_1, \ldots, \mathbf{\lambda}_m, \mathbf{\mu})
     &=  \sum_{i=1}^m \int_\mathcal{X} \left( \frac{1}{m} \phi_i(x) - \sum_{k=1}^K \lambda_{i,k}^\ell {\psi_k(x)} + \sum_{k=1}^K \lambda_{i,k}^u {\psi_k(x)} -\mu_i  \right) dP_i(x) \nonumber\\
     &\quad -1 + \sum_{i=1}^m \sum_{k=1}^K \lambda_{i,k}^\ell a_{i,k} - \sum_{i=1}^m  \sum_{k=1}^K \lambda_{i,k}^\ell b_{i,k} 
     +  \sum_{i=1}^m \mu_i,
\end{align}
where $\mathbf{\lambda}_i = (\lambda^\ell_{i,1}, \ldots, \lambda^\ell_{i,K}, \lambda^u_{i,1}, \ldots, \lambda^u_{i,K}) $ for $i \in [m]$ and $\mathbf{\mu} = (\mu_1, \ldots, \mu_m)$ are the dual variables. 
Thus, the dual function is 
\begin{equation}\label{eq:dual_function_infinite}
    \begin{split}
        \inf_{P_i \in \mathcal{M}} L(P_1, \ldots, P_m; \mathbf{\lambda_1}, \ldots, \mathbf{\lambda_m}, \mathbf{\mu}) 
        &= - 1 + \sum_{i=1}^m\sum_{k=1}^K \lambda_{i,k}^\ell a_{0,k}  - \sum_{i=1}^m\sum_{k=1}^K \lambda_{i,k}^u b_{0,k}  + \sum_{i=1}^m \mu_i,
    \end{split}
\end{equation}
with 
\begin{align}
        &\frac{1}{m}\phi_{i}(x) - \sum_{k=1}^K \lambda^\ell_{i,k} \psi_{k}(x) + \sum_{k=1}^K \lambda^u_{i,k}\psi_{k}(x) - \mu_i \geq 0, \;\; \forall x \in \mathcal{X}, 
        \;\; i \in [m].  \nonumber\\
            & \sum_{i=1}^m \phi_{i}(x) = 1, \;\;
            \forall x \in \mathcal{X}.\nonumber
\end{align}
We can establish strong duality using the Shapiro duality conditions for conic linear programs (Prop 3.4 in \cite{shapiro2001duality}). Indeed, the Slater type condition is satisfied using the empirical distribution in the uncertainty sets. Thus, using the strong dual for the inner maximization problem, we obtain the reformulation in Theorem \ref{thm:m_ary_infinite_dual}.
\end{proof}

The dual formulation in \eqref{eq:minmax_infinite_dual_m_ary} is an infinite dimensional optimization problem (in the variable $\phi$) with infinitely many constraints. We propose a tractable finite dimensional optimization problem as an approximation to the minmax problem, and construct a robust detection test based on the solution to the approximation. In addition, we bound the error arising from the approximation of the original problem formulation. 

Recall that $\mathcal{X}$ is a compact set in $\mathbb{R}^d$. For the sake of simplicity 
and without loss of generality, we assume that $\mathcal{X} \subseteq [0,1]^d$. First, note that the moment defining functions $\psi_k, \; k \in [K]$ are continuous functions on a compact set. Thus, they are Lipschitz functions 
with constants $L_1, \ldots, L_K$, and without loss of generality, we set $L = \max\limits_{k} L_k = 1$, since the moment functions can be scaled appropriately. In addition, we consider values of $\eta \in [\Delta,\eta_0]$, where $\eta_0 < \eta_\mathrm{max}$, 
and $0 < \Delta < \eta_0$. Let $\epsilon > 0$ such that $\eta + \epsilon \leq \eta_0$. Consider a discretization of the space $\mathcal{X}$ through an $\epsilon$-net or a covering set. Indeed, we can consider a simple and efficient construction by considering a grid of equally spaced $N = \lceil \frac{1}{\epsilon^d} \rceil$ points $\mathcal{S}_N = \{z_1, \ldots, z_N \}$ such that for any $x \in \mathcal{X}$, 

\begin{equation}
    \min_{j \in [N]} \|z_j - x\| \leq \epsilon.
\end{equation}

Here, $N$ depends on $\epsilon$, but we drop the dependence on $\epsilon$ in the notation for $N$ for readability. Let $\mathcal{P}_N$ denote all the distributions that are supported on the set $\mathcal{S}_N$. Define the relaxed uncertainty sets as follows for $i \in [m]$:
\begin{equation}\label{eq:relaxed_sets}
\begin{split}
    &\mathcal{P}^{\eta + \epsilon}_{i,N} =  \left\{ P \in \mathcal{P}_N:  \left| E_{P}[\psi_k] - E_{\hat{Q}_i}[{\psi_k}] \right| \leq \eta + \epsilon, \;\;  k \in [K] \right\}. 
\end{split}
\end{equation}

Consider the relaxation of the minmax problem in \eqref{eq:minmax_infinite_m_ary} with the uncertainty sets $\mathcal{P}_{i}^{\eta}$ replaced by $\mathcal{P}_{i,N}^{\eta + \epsilon}$ as follows:
\begin{equation}\label{eq:minmax_infinite_m_ary_relaxation}
    \begin{split}
         \inf_{\phi_{i,N} \in [0,1]^{N}, i \in [m] } &\sup_{P_{i,N} \in [0,1]^N, i \in [m]} 1 - \frac{1}{m} \sum_{i=1}^m \phi_{i,N}^T P_{i,N}   \\ 
        \text{s.t.} \;\; & \left| \sum_{j=1}^N p_{i,N}(z_j) \psi_k (z_j) - E_{\hat{Q}_i}[{\psi_k}] \right| \leq \eta + \epsilon,\;\; k \in [K],\\
        & P_{i,N}^T \mathbf{1} = 1, \;\;\; i \in [m], \\
        & \sum_{i=1}^m \phi_{i,N}(z_j) = 1, 
             \;\;\; j \in [N].
    \end{split}
\end{equation}

We can construct a tractable dual to the above problem similar to the finite alphabet case as follows:
\begin{equation}\label{eq:minmax_infinite_relaxation_dual_m_ary}
        \begin{split}
            &\inf_{\substack{\phi_{i,N} \in [0,1]^{N}, i \in [m], \\ \lambda^\ell_{i,k},\lambda^u_{i,k} \geq 0, k \in [K], i\in [m] \\ \mu_1, \ldots, \mu_m \in \mathbb{R}}} 1 - \sum_{i=1}^m\sum_{k=1}^K \lambda_{i,k}^\ell (a_{0,k} - \epsilon)   + \sum_{i=1}^m\sum_{k=1}^K \lambda_{i,k}^u (b_{0,k} + \epsilon) - \sum_{i=1}^m \mu_i, \\
            \text{s.t.} \;\; & \frac{1}{m}\phi_{i,N}(z_j) - \sum_{k=1}^K \lambda^\ell_{i,k} \psi_{k}(z_j) + \sum_{k=1}^K \lambda^u_{i,k}\psi_{k}(z_j) - \mu_i \geq 0, \;\; j \in [N],  \;\; i \in [m], \\
            & \sum_{i=1}^m \phi_{i,N}(z_j) = 1, \;\; j \in [N],
        \end{split}
    \end{equation}
    where $a_{i,k} = E_{\hat{Q}_i}[{\psi_k}] - \eta$, $b_{i,k} = E_{\hat{Q}_i}[{\psi_k}] + \eta$.
    This is a finite dimensional optimization problem and can be solved efficiently. With the optimal solution to the above problem $\phi_N^*$, we can obtain the saddle point distributions $(P_{1,N}^*, \ldots, P_{m,N}^*)$ as in the finite alphabet case.

    Let the optimal values of the minmax problem in \eqref{eq:minmax_infinite_m_ary} and the approximation problem in \eqref{eq:minmax_infinite_m_ary_relaxation} be denoted by $\gamma$ and $\gamma_{N,\epsilon}$, respectively. We have the following result showing the convergence of $\gamma_{N,\epsilon}$ to $\gamma$, and bounding the error due to the approximation.

    \begin{theorem}\label{thm:convergence}
    With the optimal values of the minmax Bayes formulation and its approximation denoted as $\gamma$ and $\gamma_{N, \epsilon}$, respectively, as $\epsilon \to 0$ (equivalently $N \to \infty$), $\gamma_{N, \epsilon}$ converges to $\gamma$, with $|\gamma_{N, \epsilon} - \gamma| \leq L_0 \epsilon$.
\end{theorem}

    \begin{proof}
        Define the function $\Gamma_\eta (\phi, \mathbf{\lambda}_1, \ldots, \mathbf{\lambda}_m, \mathbf{\mu})$ and $\Gamma_{N}^{\eta + \epsilon}(\phi_N, \mathbf{\lambda}_1, \ldots, \mathbf{\lambda}_m, \mathbf{\mu})$ as the objective functions of \eqref{eq:minmax_infinite_dual_m_ary} and \eqref{eq:minmax_infinite_relaxation_dual_m_ary} respectively as:
        \begin{align}\label{eq:define_Gammas}
        \Gamma^\eta(\phi, \mathbf{\lambda}_1, \ldots, \mathbf{\lambda}_m, \mathbf{\mu}) &=  1 - \sum_{i=1}^m\sum_{k=1}^K \lambda_{i,k}^\ell a_{0,k} + \sum_{i=1}^m\sum_{k=1}^K \lambda_{i,k}^u b_{0,k}  - \sum_{i=1}^m \mu_i, \\
        \Gamma_N^{ \eta + \epsilon}(\phi_N, \mathbf{\lambda}_1, \ldots, \mathbf{\lambda}_m, \mathbf{\mu})&= 1 - \sum_{i=1}^m\sum_{k=1}^K \lambda_{i,k}^\ell (a_{0,k} - \epsilon)   + \sum_{i=1}^m\sum_{k=1}^K \lambda_{i,k}^u (b_{0,k} + \epsilon) - \sum_{i=1}^m \mu_i.      
    \end{align}
    Consider a solution $(\phi_N^*, \mathbf{\lambda}_1^*, \ldots, \mathbf{\lambda}_m^*, \mathbf{\mu}^*)$ of the dual of the approximation to the minmax problem \eqref{eq:minmax_infinite_relaxation_dual_m_ary}, i.e., 
     \begin{align}
         \gamma_{N,\epsilon} &= \Gamma_N^{ \eta + \epsilon}(\phi_N^*, \mathbf{\lambda}_1^*, \ldots, \mathbf{\lambda}_m^*, \mathbf{\mu}^*) \nonumber\\
         &= 1 - \sum_{i=1}^m\sum_{k=1}^K \lambda_{i,k}^{*\ell} (a_{0,k} - \epsilon)   + \sum_{i=1}^m\sum_{k=1}^K \lambda_{i,k}^{*u} (b_{0,k} + \epsilon) - \sum_{i=1}^m \mu^*_i  
     \end{align}
     and 
     \begin{align}\label{eq:approximation_constraints}
         & \frac{1}{m}\phi_{i,N}^*(z_j) - \sum_{k=1}^K \lambda^{*\ell}_{i,k} \psi_{k}(z_j) + \sum_{k=1}^K \lambda^{*u}_{i,k}\psi_{k}(z_j) - \mu^*_i \geq 0, \;\; j \in [N], \;\; i \in [m],   \nonumber\\
            & \sum_{i=1}^m \phi^*_{i,N}(z_j) = 1, \;\; j \in [N].
     \end{align}
     
     Now define a $\phi^* : \mathcal{X} \to \Delta^m$  as $\phi^*(x) = \phi_N^*(z_x)$ for $z_x \in \mathop{\arg\min}\limits_{z_j \in \mathcal{S}_N}\|x - z_j\|$. We observe that for any $x \in \mathcal{X}$, with $z_x \in \mathop{\arg\min}\limits_{z_j \in \mathcal{S}_N}\|x - z_j\|$,
    \begin{align}
        &\frac{1}{m}\phi_i^*(x) - \sum_{k=1}^K \lambda^{*\ell}_{i,k} \psi_{k}(x) + \sum_{k=1}^K \lambda^{*u}_{i,k}\psi_{k}(x) - \mu^*_i + \epsilon \mathbf{1}^T \mathbf{\lambda_i^*}\nonumber \\
        &\geq \frac{1}{m} \phi_{i,N}^*(z_x)  - \sum_{k=1}^K \lambda^{*\ell}_{i,k} (\psi_{k}(z_x) + \epsilon) +  \sum_{k=1}^K \lambda^{*u}_{i,k}(\psi_{k}(z_x) - \epsilon) - \mu^*_i + \epsilon \sum_{k=1}^K (\lambda^{*\ell}_{i,k} + \lambda^{*u}_{i,k}) \nonumber \\ 
        &=  \frac{1}{m} \phi_{i,N}^*(z_x) - \sum_{k=1}^K \lambda^{*\ell}_{i,k} \psi_{k}(z_x) +  \sum_{k=1}^K \lambda^{*u}_{i,k} \psi_{k}(z_x) - \mu_i^*\nonumber \\
        &\geq 0 \label{eq:feasible_point_2},
    \end{align}
    where the first inequality follows from the Lipschitz property of the functions $\psi_k(.)$ and $\phi^*(x) = \phi_N^*(z_x)$, and \eqref{eq:feasible_point_2} follows since $z_x \in \mathcal{S}_N$ and \eqref{eq:approximation_constraints}.

    Thus, the vector $(\phi^*, \mathbf{\lambda}_1^*, \ldots, \mathbf{\lambda}_m^*, \mathbf{\mu}^* - \epsilon \mathbf{1}^T\mathbf{\lambda})$, where $\mathbf{\lambda} = (\mathbf{\lambda}_1^*, \ldots, \mathbf{\lambda}_m^*)$, is a feasible point for the infinite dimensional dual problem in \eqref{eq:minmax_infinite_dual_m_ary}, and
    \begin{align}
         \gamma &\leq \Gamma^\eta(\phi^*, \mathbf{\lambda}_1^*, \ldots, \mathbf{\lambda}_m^*, \mathbf{\mu}^* - \epsilon \mathbf{1}^T\mathbf{\lambda}) \nonumber\\
        &= \Gamma_N^{ \eta + \epsilon}(\phi_N^*, \mathbf{\lambda}_1^*, \ldots, \mathbf{\lambda}_m^*, \mathbf{\mu}^*)\nonumber\\
        &= \gamma^D_{N,\epsilon}, \label{eq:eqaulity_of_objectives}
    \end{align}
    where the last step follows from the definitions of the functions $\Gamma^\eta$ and $\Gamma_N^{ \eta + \epsilon}$.

     Consider the original minmax problem with an expanded radius around the uncertainty sets as
    \begin{equation}\label{eq:primal_infinte_expanded}
    \begin{split}
         \inf_{\phi} &\sup_{P_i \in \mathcal{P}^{\eta + \epsilon}_i, i \in [m]}  1 - \sum_{i=1}^m\int_\mathcal{X} \phi_i(x) dP_i(x) \\ 
        \text{s.t.} \;\; & \left| \int_\mathcal{X}  \psi_k (x) dP_i(x) - E_{\hat{Q}_i}[{\psi_k}] \right| \leq \eta + \epsilon,\;\; k \in [K]\\
        & \int_\mathcal{X} dP_i(x) = 1, \;\; i \in [m],   
    \end{split}
\end{equation}
with optimal value $\gamma_\epsilon$. 
Note that 
\begin{align}
    \gamma_\epsilon &= \inf_{\phi} \sup_{P_i \in \mathcal{P}^{\eta + \epsilon}_i, i \in [m]}  P_E(\phi; P_1, \ldots,  P_m) \nonumber\\
    &= \sup_{P_i \in \mathcal{P}^{\eta + \epsilon}_i, i \in [m]}  \inf_{\phi} P_E(\phi; P_1, \ldots,  P_m) \label{eq:comparison_infinite_approx_objective_1} \\
    &\geq \sup_{P_{i,N} \in \mathcal{P}^{\eta + \epsilon}_{i,N}, i \in [m]}  \inf_{\phi_N} P_E(\phi_N; P_{1,N}, \ldots. P_{m,N}) \label{eq:comparison_infinite_approx_objective_2}\\
    &= \inf_{\phi_N} \sup_{P_{i,N} \in \mathcal{P}^{\eta + \epsilon}_{i,N}, i \in [m]}    P_E(\phi_N; P_{1,N}, \ldots, P_{m,N}) = \gamma_{N,\epsilon},\nonumber
\end{align}
where the first inequality follows from the fact that the feasibility sets $\mathcal{P}^{\eta + \epsilon}_i$ in \eqref{eq:comparison_infinite_approx_objective_1} are larger than the feasibility sets in \eqref{eq:comparison_infinite_approx_objective_2}. 
Thus, it is clear to see that
    \begin{equation}\label{eq:gamma_ineq}
        \gamma_\epsilon \geq \gamma_{N,\epsilon} \geq \gamma.
    \end{equation}

We will now show that 
    \begin{equation}
        \lim_{\epsilon \to 0} \gamma_\epsilon = \gamma.
    \end{equation}
    Note that the dual of \eqref{eq:primal_infinte_expanded} can be constructed as in Theorem \ref{thm:m_ary_infinite_dual} as follows:
     \begin{equation}\label{eq:minmax_infinite_expanded_dual}
        \begin{split}
            \inf_{\substack{\phi, \\ \lambda^\ell_{i,k},\lambda^u_{i,k} \geq 0, k \in [K], i\in [m] \\ \mu_1, \ldots, \mu_m \in \mathbb{R}}} &1 - \sum_{i=1}^m\sum_{k=1}^K \lambda_{i,k}^\ell (a_{0,k} - \epsilon) + \sum_{i=1}^m\sum_{k=1}^K \lambda_{i,k}^u (b_{0,k} + \epsilon)  - \sum_{i=1}^m \mu_i, \\
            \text{s.t.} \;\; & \frac{1}{m}\phi_i(x) - \sum_{k=1}^K \lambda^\ell_{i,k} \psi_{k}(x) + \sum_{k=1}^K \lambda^u_{i,k}\psi_{k}(x) - \mu_i \geq 0, \;\; \forall x \in \mathcal{X},  \;\;\; i \in [m],   \\
            & \sum_{i=1}^m \phi_i(x) = 1, i \in [m]. 
        \end{split}
    \end{equation}

    From the definition of \eqref{eq:define_Gammas}, we have that  
    \begin{equation}
    \begin{split}
    \Gamma^{\eta + \epsilon}(\phi_N, \mathbf{\lambda}_1, \ldots, \mathbf{\lambda}_m, \mathbf{\mu}) &= 1 - \sum_{i=1}^m\sum_{k=1}^K \lambda_{i,k}^\ell (a_{0,k} - \epsilon) + \sum_{i=1}^m\sum_{k=1}^K \lambda_{i,k}^u (b_{0,k} + \epsilon)  - \sum_{i=1}^m \mu_i    
    \end{split}
     \end{equation}
   is the objective of the dual of the extended problem in \eqref{eq:minmax_infinite_expanded_dual}. Let $\Phi, \Lambda_1, \ldots, \Lambda_m, \mathcal{U}$ denote the feasibility sets for the variables $\phi, \mathbf{\lambda}_1, \ldots, \mathbf{\lambda}_m, \mathbf{\mu}$. Note that the feasibility sets for the dual of the original problem in \eqref{eq:minmax_infinite_dual_m_ary} are the same.
   Define
   \begin{equation}
       g(\eta) = \inf_{\substack{\phi \in \Phi, \\ \mathbf{\lambda_i} \in \Lambda_i, i \in [m], \\ \mathbf{\mu} \in \mathcal{U}}} \Gamma^{\eta}(\phi, \mathbf{\lambda}_1, \ldots, \mathbf{\lambda}_m, \mathbf{\mu}).
   \end{equation}
    Note that 
$
            \gamma_{\epsilon} = g(\eta + \epsilon) 
$
    and 
$
            \gamma = g(\eta).
    $
    We will show that $g(\eta)$ is a convex function in $\eta$. Indeed, let $\mathbf{d}^1 = (\phi^1, \mathbf{\lambda}_1^1, \ldots, \mathbf{\lambda}_m^1, \mathbf{\mu}^1)$ and $\mathbf{d}^2 = (\phi^2, \mathbf{\lambda}_1^2, \ldots, \mathbf{\lambda}_m^2, \mathbf{\mu}^2)$ be solutions to $g(\eta^1)$ and $g(\eta^2)$, respectively. Then, for $0 < \alpha < 1$, it is easy to see that $\alpha\mathbf{d}^1 + (1-\alpha)\mathbf{d}^2$ is a feasible solution for the minimization in $g(\alpha\eta^1 + (1-\alpha)\eta^2)$ as well, and 
    \begin{align}
        &g(\alpha\eta^1 + (1-\alpha)\eta^2) \nonumber \\ 
        &\leq 1 - \sum_{i=1}^m\sum_{k=1}^K (\alpha{\lambda^1}_{0,k}^\ell + (1-\alpha){\lambda^2}_{0,k}^\ell) a_{0,k} + \sum_{i=1}^m \sum_{k=1}^K (\alpha{\lambda^1}_{0,k}^u + (1-\alpha){\lambda^2}_{0,k}^u)b_{0,k}
        \nonumber \\  
        & - \sum_{i=1}^m \sum_{k=1}^K (\alpha{\lambda^1}_{1,k}^\ell + (1-\alpha){\lambda^2}_{1,k}^\ell) a_{1,k} + \sum_{i=1}^m \sum_{k=1}^K (\alpha{\lambda^1}_{1,k}^\ell + (1-\alpha){\lambda^2}_{1,k}^\ell) b_{1,k}  \nonumber \\
        &- \alpha\sum_{i=1}^m{\mu}_i^1 + (1-\alpha)\sum_{i=1}^m{\mu}_i^2\nonumber \\
        &= \alpha g(\eta^1) + (1 - \alpha)g(\eta^2).
    \end{align}
    Thus, $g(\eta)$ is convex in $\eta$ on the open interval $(0, \eta_\mathrm{max})$, and hence Lipschitz on the closed interval $[\Delta, \eta_0] \subset (0, \eta_\mathrm{max})$. This implies that $g(\eta)$ is continuous on the interval $[\Delta, \eta_0]$, and hence $\lim_{\epsilon \to 0} g(\eta + \epsilon) = g(\eta)$, or equivalently, 
    \begin{equation}
        \lim_{\epsilon \to 0} \gamma_{\epsilon} = \gamma,
    \end{equation}
    and thus, 
    \begin{equation}
        \lim_{\epsilon \to 0} \gamma_{N,\epsilon} = \gamma.
    \end{equation}
    Additionally, from \eqref{eq:gamma_ineq}, the error due to the approximation can be bounded as:
    \begin{align}
        \gamma_{N,\epsilon} - \gamma &\leq \gamma_{\epsilon} - \gamma \leq L_0 \epsilon,
    \end{align}
    where $L_0$ is the Lipschitz constant of the function $g(\eta)$.    
    \end{proof}

    We have proved the convergence of the optimal value of the approximation problem $\gamma_{N,\epsilon}$ to the minmax Bayes error $\gamma$. We will now construct a robust detection test from the solution of the approximation problem in \eqref{eq:minmax_infinite_relaxation_dual_m_ary}. Note that $P^*_{1,N}, \ldots, P^*_{m,N}$ are the solutions to the maximization problem in \eqref{eq:minmax_infinite_m_ary_relaxation} with the optimal $\phi_N^*$ from its corresponding dual problem, i.e., $(\phi_N^*, P^*_{1,N}, \ldots, P^*_{m,N})$ is a saddle point for the relaxed problem in \eqref{eq:minmax_infinite_m_ary_relaxation}. Recall the partition $\{\mathcal{A}_1, \ldots,\mathcal{A}_N\}$ defined by the set $\mathcal{S}_N$ on $\mathcal{X}$ such that for any $j \in [N]$, if $x \in \mathcal{A}_j$, then
\begin{equation}
    \|x - z_j\| \leq \epsilon.
\end{equation}
    
    In order to construct a robust detection test, we extend these discrete distributions to the whole space $\mathcal{X}$ as $P_1^*, \ldots, P_m^*$. This can be done through kernel smoothing as:
    \begin{equation}
        p_i^*(x) = \sum_{j=1}^N k(x,z_j) p_{i,N}^*(z_j).
    \end{equation}
    Popular kernels used for smoothing are k-Nearest Neighbor (kNN) kernels, Gaussian kernels and the Epanechnikov (parabola) kernel. 
    Define the pair-wise likelihood ratios as follows:
    \begin{equation}
        \ell_i(x) = \frac{p_i^*(x)}{p_1^*(x)}.
    \end{equation}
    The proposed test for the Bayesian m-ary robust testing problem is defined as follows. For $x \in \mathcal{X}$, if $\mathop{\arg\max}\limits_{i \in [m]} \ell_i(x) = \{i^*\}$ (i.e., there are no tie breaks), then
    \begin{equation}\label{eq:m-ary_robust_test}
        \phi_{i}^*(x) = 
        \begin{cases}
            1 \;\;\; \text{ if } i = i^* \\
            0 \;\;\; \text{ otherwise}.
        \end{cases}
    \end{equation}
    In case of a tie, i.e., $\mathop{\arg\max}\limits_{i \in [m]} \ell_i(x)$ is not a singleton set, $\phi^*(x) = \phi_N^*(z_x)$ for $z_x \in \mathop{\arg\min}\limits_{z_j \in \mathcal{S}_N}\|x - z_j\|$.

    For instance, the 1-NN kernel smoothing test can be defined as follows. The mass $p^*_{i,N}(z_j)$ is distributed uniformly on all points in the corresponding set in the partition $\mathcal{A}_j$ for $j \in [N]$, i.e., for $x \in \mathcal{X}$
    \begin{equation}\label{eq:extend_distributions}
         p_i^*(x) = \sum_{j=1}^N \frac{p^*_{i,N}(z_j) \identityf{x \in \mathcal{A}_j}}{\int_{\mathcal{A}_j} dx}.
    \end{equation}
    In this case, the proposed test reduces to a 1-NN test, where for $x \in \mathcal{X}$
    \begin{equation}\label{eq:1NN_test}
        \phi^*(x) = \phi_{N}^*(z_x), \;\;\; z_x \in \mathop{\arg\min}\limits_{z_j \in \mathcal{S}_N}\|x - z_j\|, 
    \end{equation}
    since $\phi_N^*$ is the likelihood ratio test with $P_{1,N}^*, \ldots, P_{m,N}^*$ when there are no tie breaks.
    
    We present the following result that bounds the worst-case error of the  1-NN kernel smoothing test in \eqref{eq:1NN_test} over all distributions in the uncertainty sets with respect to the optimal minmax Bayes error. 
   \begin{theorem}\label{thm:robust_test}
    With the optimal value of the minmax Bayes formulation as $\gamma$, let the robust test $\phi^*$ with 1-NN kernel smoothing be as defined in \eqref{eq:1NN_test}. Then, 
    \begin{equation}
        \sup_{P_i \in \mathcal{P}_i^\eta, i \in [m]} P_E(\phi^*; P_1, \ldots, P_m) - \gamma \leq L_0 \epsilon.
    \end{equation}
\end{theorem}

    \begin{proof}
        Recall that
    \begin{align}
        \gamma_{N,\epsilon} &= \inf_{\phi_N}\sup_{P_{i,N} \in \mathcal{P}^{\eta + \epsilon}_{i,N}, i \in [m]} P_E(\phi_N; P_{1,N}, \ldots, P_{m,N})\nonumber \\
        &= \sup_{P_{i,N} \in \mathcal{P}^{\eta + \epsilon}_{i,N}, i \in [m]} \inf_{\phi_N} P_E(\phi_N; P_{1,N}, \ldots, P_{m,N}).
    \end{align}
    Since $(\phi_N^*, P^*_{1,N}, \ldots,  P^*_{m,N})$ is a saddle point for the above optimization problem, for any  $P_{i,N}\in \mathcal{P}^{\eta + \epsilon}_{i,N}, i \in [m]$ and any $\phi_N$,
    \begin{align}\label{eq:saddle_point_relaxation}
        P_E(\phi^*_N; P_{1,N}, \ldots, P_{m,N}) &\leq P_E(\phi^*_N; P^*_{1,N}, \ldots, P^*_{m,N}) \nonumber\\
        &\leq P_E(\phi_N; P^*_{1,N}, \ldots, P^*_{m,N}).
    \end{align}

    Consider the distributions $P_1^*, \ldots, P_m^*$, and the decision rule $\phi^*$ extended to the entire alphabet $\mathcal{X}$ as in \eqref{eq:extend_distributions} and \eqref{eq:1NN_test}, respectively. 
    For any $P_i \in \mathcal{P}_i^\eta, i \in [m]$, we can construct $P_{1,N}, \ldots,  P_{m,N}$ on $\mathcal{S}_N$ as follows:
    \begin{align}
        p_{i,N}(z_j) &= P_i(\mathcal{A}_j),
    \end{align}
    where $\mathcal{A}_j$ is the partition containing the point $z_j$. Then for $i \in [m]$ and any $k \in [K]$, 
    \begin{align}
        &\left| \sum_{j=1}^N p_{i,N}(z_j) \psi_k (z_j) - E_{\hat{Q}_i}[{\psi_k}] \right| \nonumber \\
        &\leq  \left| E_{P_i}[{\psi_k}] - E_{\hat{Q}_i}[{\psi_k}]\right| +  \left| \sum_{j=1}^N p_{i,N}(z_j) \psi_k (z_j) - E_{P_i}[{\psi_k}] \right|\nonumber \\
        &\leq \eta + \epsilon, 
    \end{align}
    where the last inequality follows from the Lipschitz property of $\psi_K$ and the fact that $P_i \in \mathcal{P}_i^\eta, i \in [m]$. Thus, we have that $P_{i,N} \in \mathcal{P}_{i,N}^{\eta + \epsilon}$. In addition, 
    \begin{align}
        E_{P_{i,N}}[\phi^*_N] &= \sum_j p_{i,N}(z_j) \phi_N^*(z_j) \nonumber\\
        &= \sum_j \int_{\mathcal{A}_j} \phi_N^*(z_j) p_i(x) dx \nonumber\\
        &= \sum_j \int_{\mathcal{A}_j} \phi^*(x) p_i(x) dx \nonumber\\
        &=  \int_{\mathcal{X}} \phi^*(x) p_i(x) dx = E_{P_i}[\phi^*].
    \end{align}
    It follows that 
    \begin{equation}
        P_E(\phi^*; P_1, \ldots, P_m) = P_E(\phi_N^*; P_{1,N}, \ldots, P_{m,N}).
    \end{equation}
    Combining the above with \eqref{eq:saddle_point_relaxation}, we get that 
    \begin{equation}
        P_E(\phi^*; P_1, \ldots, P_m) \leq P_E(\phi^*_N; P^*_{1,N}, \ldots, P^*_{m,N}) = \gamma_{N,\epsilon}.
    \end{equation}
    Thus, using the result from Theorem \ref{thm:convergence}, for any $P_i \in \mathcal{P}_i^\eta, i \in [m]$, 
    \begin{equation}
        P_E(\phi^*; P_1, \ldots, P_m) \leq \gamma + L_0 \epsilon.
    \end{equation}
    Taking supremum over all distributions in the uncertainty set, we have that 
    \begin{equation}
        \sup_{P_i \in \mathcal{P}_i^\eta, i \in [m]} P_E(\phi^*; P_1, \ldots, P_m) \leq \gamma + L_0 \epsilon.
    \end{equation}
    \end{proof}

\begin{remark} The key idea we used in the proof of Theorem~\ref{thm:robust_test} is that the 1-NN test at any point $x$ is the same as that at the point $z_x$ in the $\epsilon$-net $\mathcal{S}_N$ closest to $x$, and thus the Lipschitz property of the moment functions can be used to bound the error probability of the 1-NN test for all distributions in the uncertainty sets. However, for other kernels such as the Gaussian kernel or a k-NN kernel with $k \geq 2$, the test at any point $x$ depends on not just $z_x$, but additional points from $\mathcal{S}_N$ as well. This makes it challenging to bound the error probability in a similar manner as in the proof of Theorem~\ref{thm:robust_test}, since the Lipschitz property of the moment functions is no longer enough to obtain such a bound.  
\end{remark}

The results provided so far are for testing a single observation $x$. As in classical multi-hypothesis testing, for a batch of i.i.d. observations $\{x_j\}_{j=1}^s$, we propose a test that chooses the hypothesis with the maximum sum of log-likelihood ratios $\sum_{j=1}^s \log{\ell_i(x_j)}$. Note that this test is not necessarily optimal; we evaluate the performance on i.i.d. samples empirically in Section~\ref{sec:exp}.

\section{Experimental Results} \label{sec:exp}
In this section, we provide some simulation results for both synthetic data and real data. Recall that for the infinite alphabet case, we need to extend the discrete distributions $P_{i, N}^*$ to the whole space to get $P_i^*$. In our experiments, we compare different kernels to extend the discrete distributions to the whole space: the k-NN kernel for different values of $k$, the Gaussian kernel and the parabola kernel. The k-NN kernel smoothing with uniform averaging is given by:
\begin{equation}
    k(x,z_j) = 
    \begin{cases}
        \frac{1}{k} \;\;\;\; \text{if } j \in \mathcal{I}_x \\
        0 \;\;\;\; \text{otherwise},
    \end{cases}
\end{equation}
where $\mathcal{I}_x$ is the set of indices of the $k$ nearest neighbours of $x$ in $\mathcal{S}_N$. The Gaussian kernel with bandwidth $h$ is given by 
\begin{equation}
    k(x,z_j) = \frac{1}{\sqrt{2\pi }h}\exp{\left(\frac{-\|x - z_j\|^2}{2h^2}\right)}.
\end{equation}
The parabola kernel with bandwidth $h$ is given by
\begin{equation}
    k(x,z_j) = 
    \begin{cases}
        0.75*(1 - \frac{\|x - z_j\|^2}{h^2}) \;\;\; \text{if } \frac{\|x - z_j\|}{h} \leq 1 \\
        0 \;\;\;\;\;\;\;\;\;\;\;\;\;\;\;\;\;\;\;\;\;\;\;\;\;\;\;\;\;\; \text{otherwise}.
    \end{cases}
\end{equation}
Additionally, we compare our proposed test with a heuristic test defined as follows. 
Let $\hat{P}_s = \frac{1}{s}\sum_{j=1}^s \delta_{{x}_{j}}$ be the empirical distribution of the batch samples $x^s = \{x_j\}_{j=1}^s$. Consider the test statistics 
\begin{align}
    T_i(x^s) = \sum_{k=1}^K \left| E_{\hat{P}_s}[\psi_k] - E_{\hat{Q}_i}[\psi_k] \right|^2. 
\end{align}
We propose a heuristic test for i.i.d. observations that chooses the hypothesis 
\begin{equation}\label{eq:batch_test}
    i^* \in \mathop{\arg\min}_{i\in [m]} T_i(x^s),
\end{equation}
with ties broken arbitrarily. 

In Appendix \ref{sec:direct}, we show that the proposed heuristic test is exponentially consistent. Therefore, it serves as a good benchmark to compare with our proposed test. 

\begin{remark}
    It is to be noted that different constructions of uncertainty sets have their own advantages and disadvantages, and choosing a particular construction depends on the application at hand. It is not meaningful to compare the empirical performances of robust tests proposed for different uncertainty sets, as it might not be possible to ascertain that the sets constructed using different methods are equivalent in some sense to ensure a fair comparison.
\end{remark}  

\subsection{Synthetic Data}
In this section, we focus on synthetic data experiments. We first compare the performance of our proposed test with the heuristic test in a setting with three hypotheses. We use different values of $k$ in the k-NN kernel smoothing method to examine the impact of $k$. For the Gaussian kernel smoothing method, we use a Gaussian kernel with bandwidth 1, and for the parabola kernel, we use a bandwidth of 0.5. We use 10 samples from the Gaussian distribution $\mathcal{N}([0, 0], [[1,0],[0,1]])$, 10 samples from the Gaussian distribution $\mathcal{N}([-1, -1], [[1,0],[0,1]])$ and 10 samples from the Gaussian distribution $\mathcal{N}([1, 1], [[1,0],[0,1]])$ as the training samples. All the training data from the three hypotheses (30 in total) are used to construct the discretization set $\mathcal{S}_N$. We then use the true distributions to evaluate the performance of the proposed algorithms. We use the coordinate-wise first and second moments of the training data to construct uncertainty sets. We plot the log of the error probability as a function of the testing sample size. It can be seen from Fig.~\ref{syn_comparison_3} that the parabola kernel smoothing method for our proposed test has the best performance. Moreover, the error probability decreases with an increase in $k$ in the $k$-nearest neighbor method. The 9-NN, 12-NN, Gaussian and parabola kernel smoothing methods perform better than the heuristic test.

\begin{figure}[htb]
	\centering 
		\includegraphics[width=0.6\linewidth]{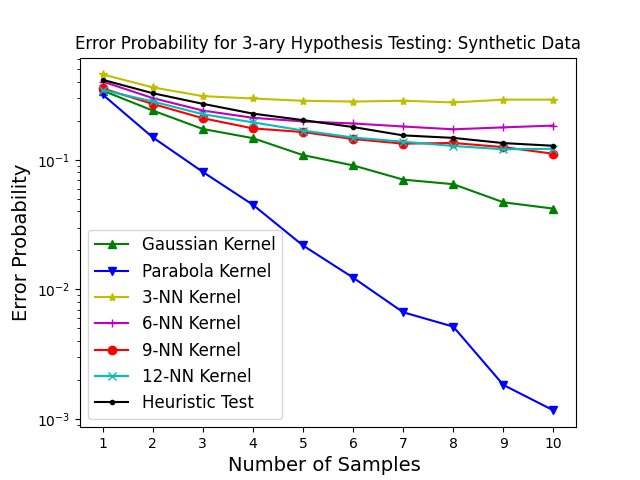}
	\caption{Comparison of the Robust Test and the Heuristic Test for Three Hypothesis: Synthetic Data}
	\label{syn_comparison_3}
\end{figure}

We now evaluate the performance of our method for four hypotheses. The first three hypotheses are taken to be the same as in the previous case, and for the uncertainty set $\mathcal{P}_4$, we use 10 samples from the Gaussian distribution $\mathcal{N}([2, 2], [[1,0],[0,1]])$ as the training samples. From Fig.~\ref{syn_comparison_4}, it can be seen that the 9-NN, 12-NN, Gaussian and parabola kernel smoothing methods  perform better than the heuristic test. We observe that for the $k$-nearest neighbor smoothed test, a small value of $k$ may not lead to good performance as the sample size increases.
\begin{figure}[htb]
	\centering 
		\includegraphics[width=0.6\linewidth]{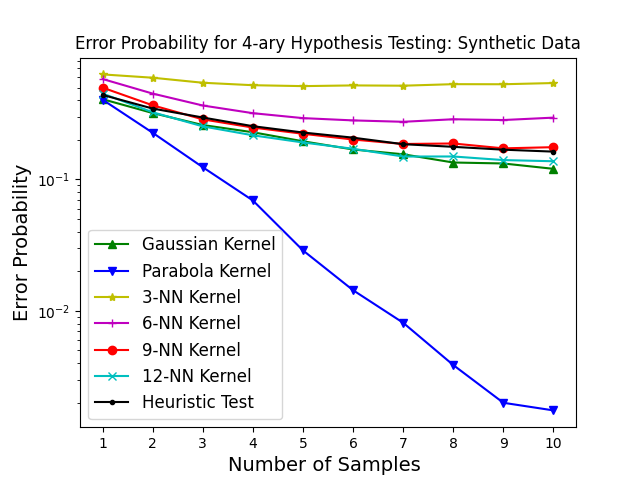}
	\caption{Comparison of the Robust Test and the Heuristic Test for Four Hypothesis: Synthetic Data}
	\label{syn_comparison_4}
\end{figure}

\subsection{Real Data}
In this section, we provide some results for real data experiments to compare the performance of different tests.  We use a dataset collected with the Actitracker system \cite{lockhart2011design, kwapisz2011activity, gary2012impact} to form the hypotheses. For the case with three hypotheses, we use 25 samples of the walking data from user 669, 25 samples of the jogging data from user 685 and 25 samples of the sitting data from user 594 to construct the uncertainty sets. 
We use the coordinate wise first and second moments as the constraints function. We plot the error probability as a function of the testing sample size. From Fig.~\ref{real_comparison_3}, we have that our proposed test with k-NN, Gaussian and parabola smoothing performs better than the heuristic test.

\begin{figure}[htb]
	\centering 
		\includegraphics[width=0.6\linewidth]{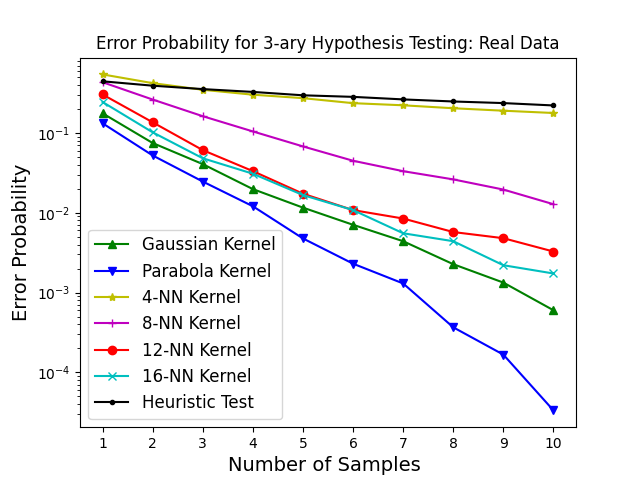}
	\caption{Comparison of the Robust Test and the Heuristic Test for Three Hypothesis: Real Data}
	\label{real_comparison_3}
\end{figure}

We also evaluate the performance of the proposed methods for four hypotheses. We use 25 samples of the walking data from user 669, 25 samples of the jogging data from user 685, 15 samples of the sitting data from user 594 and 15 samples of the lying down data from user 1603 to construct the uncertainty sets. We use the coordinate-wise first and second moments as the constraints function. In Fig.~\ref{real_comparison_4}, we plot the error probability as a function of the testing sample size. It can be seen that the $k$-NN, Gaussian and parabola smoothed tests perform better than the heuristic test.
\begin{figure}[htb]
	\centering 		\includegraphics[width=0.6\linewidth]{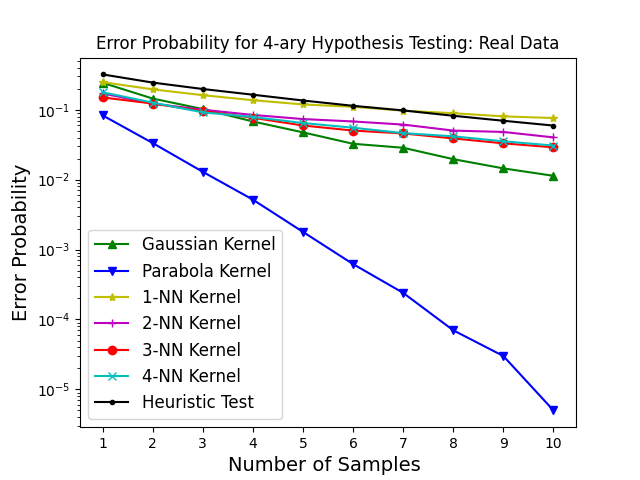}
	\caption{Comparison of the Robust Test and the Heuristic Test for Four Hypothesis: Real Data}
	\label{real_comparison_4}
\end{figure}

\section{Conclusion}

In this paper, we studied the robust multi-hypothesis ($m \geq 2$) testing problem with moment constrained uncertainty sets. In the case of robust binary hypothesis testing, we provided a counter-example to show that the Joint Stochastic Boundedness property might not hold for moment-constrained uncertainty sets, which thus  motivates the study of optimization based approaches. For the general setting with $m \geq 2$, we characterized the optimal minmax test in the finite alphabet case. In the infinite alphabet case, we proposed an efficient finite-dimensional minmax reformulation to the intractable infinite-dimensional optimization problem through a discretization approach, and provided convergence guarantees of the optimal value of the reformulation to the original minmax problem. We proposed a robust test constructed from smoothing the discrete saddle-point worst-case distributions of the reformulation problem to the entire space. For the 1-NN smoothed test, we provided guarantees on the worst-case error of the proposed test over all distributions in the uncertainty sets. Through experiments, we demonstrated the exponential consistency of our proposed test for different smoothing methods such as the k-NN kernel, Gaussian kernel and the parabola kernel.

\bibliographystyle{IEEEtran}
\bibliography{ref.bib}

\appendix
\subsection{Exponential Consistency of the Heuristic Test}\label{sec:direct}
\begin{theorem}
    The test in \eqref{eq:batch_test} is exponentially consistent. 
\end{theorem}

\begin{proof}
We consider the error probability under hypothesis $H_1$. The following results can be easily generalized to all $m$ hypotheses.

For any $P_1 \in \mathcal{P}^\eta_1$, we have that for any $i = 2, \cdots, m$,
\begin{align}
    &P_{1} \left( \sum_{k=1}^K \left| E_{\hat{P}_s}[\psi_k] - E_{\hat{Q}_1}[\psi_k] \right|^2 \geq  \sum_{k=1}^K \left| E_{\hat{P}_s}[\psi_k] - E_{\hat{Q}_i}[\psi_k] \right|^2 \right) \nonumber\\
    &= P_{1}\left(2  \sum_{k=1}^K (E_{\hat{Q}_i}[\psi_k] - E_{\hat{Q}_1}[\psi_k])E_{\hat{P}_s}[\psi_k] \geq \sum_{k=1}^K (E_{\hat{Q}_i}^2[\psi_k] - E_{\hat{Q}_1}^2[\psi_k])  \right) \nonumber\\
    &= P_{1}\left( F_i(x_1, \ldots, x_s) - E[F_i(x_1, \ldots, x_s)] \geq \sum_{k=1}^K (E_{\hat{Q}_i}[\psi_k] - E_{{P}_1}[\psi_k])^2 - (E_{\hat{Q}_1}[\psi_k] - E_{{P}_1}[\psi_k])^2  \right),
\end{align}
where $F_i(x_1, \ldots, x_s) = 2  \sum_{k=1}^K (E_{\hat{Q}_i}[\psi_k] - E_{\hat{Q}_1}[\psi_k])E_{\hat{P}_s}[\psi_k] $. Recall that
\begin{equation}
    \max\limits_{k=1,\ldots,K} \sup\limits_{x \in \mathcal{X}} \psi_k(x) \leq M.
\end{equation}
Thus, $F_i(x_1, \ldots, x_s)$ satisfies the bounded differences property, i.e., 
\begin{align}
    &\sup\limits_{x \in \mathcal{X}} \left| F_i(x_1, \ldots, x_i, \ldots, x_s) - F_i(x_1, \ldots, x_i', \ldots, x_s) \right| \nonumber\\
    &=  \sup\limits_{x \in \mathcal{X}} \left| \frac{2}{s} \sum_{k=1}^K (E_{\hat{Q}_i}[\psi_k] - E_{\hat{Q}_1}[\psi_k]) (\psi_k(x_i) - \psi_k(x_i')) \right| \nonumber\\
    &\leq \frac{4M}{s} \sum_{k=1}^K \left| E_{\hat{Q}_i}[\psi_k] - E_{\hat{Q}_1}[\psi_k] \right|.
\end{align}
Also, since $P_1 \in \mathcal{P}^\eta_1$, 
\begin{equation}
    \sum_{k=1}^K (E_{\hat{Q}_i}[\psi_k] - E_{{P}_1}[\psi_k])^2 - (E_{\hat{Q}_1}[\psi_k] - E_{{P}_1}[\psi_k])^2 > 0.
\end{equation}
Thus, using Mcdiarmid's inequality \cite{mcdiarmid1989}, 
\begin{align}
    &P_{1} \left( \sum_{k=1}^K \left| E_{\hat{P}_s}[\psi_k] - E_{\hat{Q}_1}[\psi_k] \right|^2 \geq  \sum_{k=1}^K \left| E_{\hat{P}_s}[\psi_k] - E_{\hat{Q}_i}[\psi_k] \right|^2 \right) \nonumber\\
    &\leq \exp{\left( \frac{-s\left( \sum_{k=1}^K (E_{\hat{Q}_i}[\psi_k] - E_{{P}_1}[\psi_k])^2 - (E_{\hat{Q}_1}[\psi_k] - E_{{P}_1}[\psi_k])^2 \right)^2}{8M^2 \left(\sum_{k=1}^K \left| E_{\hat{Q}_i}[\psi_k] - E_{\hat{Q}_1}[\psi_k] \right| \right)^2} \right)}.
\end{align}
Thus, we have that 
\begin{align}
    &\sup\limits_{P_1 \in \mathcal{P}^\eta_1} P_{1} \left( \sum_{k=1}^K \left| E_{\hat{P}_s}[\psi_k] - E_{\hat{Q}_1}[\psi_k] \right|^2 \geq  \sum_{k=1}^K \left| E_{\hat{P}_s}[\psi_k] - E_{\hat{Q}_i}[\psi_k] \right|^2 \right) \nonumber\\
    &\leq \sup\limits_{P_1 \in \mathcal{P}^\eta_1} \exp{\left( \frac{-s\left( \sum_{k=1}^K (E_{\hat{Q}_i}[\psi_k] - E_{{P}_1}[\psi_k])^2 - (E_{\hat{Q}_1}[\psi_k] - E_{{P}_1}[\psi_k])^2 \right)^2}{8M^2 \left(\sum_{k=1}^K \left| E_{\hat{Q}_i}[\psi_k] - E_{\hat{Q}_1}[\psi_k] \right| \right)^2} \right)} \nonumber\\
    &\leq \exp{\left( \frac{-s\left( \sum_{k=1}^K \left( E_{\hat{Q}_i}[\psi_k] - E_{\hat{Q}_1}[\psi_k] + 2\eta \right) \left( E_{\hat{Q}_i}[\psi_k] - E_{\hat{Q}_1}[\psi_k] \right) \right)^2}{8M^2 \left(\sum_{k=1}^K \left| E_{\hat{Q}_i}[\psi_k] - E_{\hat{Q}_1}[\psi_k] \right| \right)^2} \right)}.
\end{align}
Therefore, we have that 
\begin{align}
&\sup\limits_{P_1 \in \mathcal{P}^\eta_1} P_{1} \left( \phi_s(x^s) \neq 1 \right) \nonumber\\
&\leq \sum_{i=2}^m \exp{\left( \frac{-s\left( \sum_{k=1}^K \left( E_{\hat{Q}_i}[\psi_k] - E_{\hat{Q}_1}[\psi_k] + 2\eta \right) \left( E_{\hat{Q}_i}[\psi_k] - E_{\hat{Q}_1}[\psi_k] \right) \right)^2}{8M^2 \left(\sum_{k=1}^K \left| E_{\hat{Q}_i}[\psi_k] - E_{\hat{Q}_1}[\psi_k] \right| \right)^2} \right)}\nonumber\\
&\leq m\max_{i = 2, \cdots, m} \exp{\left( \frac{-s\left( \sum_{k=1}^K \left( E_{\hat{Q}_i}[\psi_k] - E_{\hat{Q}_1}[\psi_k] + 2\eta \right) \left( E_{\hat{Q}_i}[\psi_k] - E_{\hat{Q}_1}[\psi_k] \right) \right)^2}{8M^2 \left(\sum_{k=1}^K \left| E_{\hat{Q}_i}[\psi_k] - E_{\hat{Q}_1}[\psi_k] \right| \right)^2} \right)}.
\end{align}
This completes the proof.
\end{proof}


 





\end{document}